\documentclass{amsart}
\usepackage{amssymb, amsmath, amsthm, graphics, comment, xspace, enumerate}
\newtheorem{thm}{Theorem}[section]

\newtheorem{lem}[thm]{Lemma}

\theoremstyle{definition}

\theoremstyle{remark}

\numberwithin{equation}{section}

\begin{document}
\title[$(\omega,{\rho})$-Periodic solutions of abstract...]{$(\omega,{\rho})$-Periodic solutions of abstract integro-differential impulsive equations on Banach space}

\author{Michal Fe\v{c}kan}
\address{Department of Mathematical Analysis and Numerical Mathematics,
Faculty of Mathematics, Physics and Informatics, Comenius University,
Slovakia}
\address{Mathematical Institute of Slovak Academy of Sciences, \v{S}tef\'{a}nikova 49, 814 73 Bratislava, Slovakia}
\email{Michal.Feckan@fmph.uniba.sk}

\author{Marko Kosti\' c}
\address{Faculty of Technical Sciences,
University of Novi Sad,
Trg D. Obradovi\' ca 6, 21125 Novi Sad, Serbia}
\email{marco.s@verat.net}

\author{Daniel Velinov}
\address{Department for Mathematics, Faculty of Civil Engineering, Ss. Cyril and Methodius University, Skopje,
Partizanski Odredi
24, P.O. box 560, 1000 Skopje, N. Macedonia}
\email{velinovd@gf.ukim.edu.mk}

{\renewcommand{\thefootnote}{} \footnote{
\\ \text{  }  \ \    {\it Key words and phrases.} $(\omega,{\rho})$-periodic solutions, periodic solutions, semilinear integro-differential impulsive equations.
\\  \text{  }  \ \ This research is partially supported by grant 174024 of Ministry
of Science and Technological Development, Republic of Serbia and Bilateral project between MANU and SANU, by the Slovak Research and Development Agency under the contract No. APVV-18-0308, and by the Slovak Grant Agency VEGA No. 1/0358/20 and No. 2/0127/20.}}

\begin{abstract}
In this paper, we investigate the existence and uniqueness of $(\omega,{\rho})$-periodic solutions for a class of the abstract impulsive integro-differential equations on Banach space.
\end{abstract}
\maketitle

\section{Introduction and preliminaries}

The class of $(\omega,c)$-periodic functions was introduced and investigated by E. Alvarez et al. in \cite{alvarez1}-\cite{alvarez2}. This type of periodicity naturally arises in the solution $y(t)$ of Mathieu's equations $y''+ay=2qcos(2t)y$. In \cite{agaoglou}, the authors studied the existence and uniqueness of $(\omega,c)$-periodic solutions for semilinear evolution equations $u'=Au+f(t,u)$ in complex Banach spaces. The notion of $(\omega,c)$-periodicity was generalized in \cite{feckan2}, by M. Fe\v{c}kan, K. Liu and J. Wang, who considered $(\omega,{\mathbb T})$-periodic solutions for this class of semilinear evolution equations, where ${\mathbb T}$ is linear isomorphism on a Banach space $X$. For some other generalizations of this concept see \cite{feckan3}.

On the other side, the impulsive differential equations describe evolution processes characterized by the fact that at certain moments they experience a change of state abruptly, i.e., these processes are subject to short-term perturbations whose duration is negligible compared with the duration of the whole process (see \cite{ahmed}-\cite{ahmed1}, \cite{bain1}-\cite{bain2}, \cite{guo}, \cite{hala}, \cite{hrbain}, \cite{bail},
\cite{li}, \cite{peng}, \cite{wang2}, \cite{wei}). Many biological phenomena involving thresholds, bursting rhythm models in medicine and biology, optimal control models in economics, pharmacokinetics and frequency modulated systems, do exhibit impulsive effects, \cite{bail}. Therefore, the interest for investigating the qualitative features of the solutions of these impulsive systems is quite big. In \cite{li1}, the $(\omega,c)$-periodic solutions of impulsive differential systems with coefficient of matrices were investigated, while the authors of \cite{liu1} used the fixed point theorems in order to clarify certain results concerning the existence and uniqueness of $(\omega,c)$-periodic solutions for nonlinear impulsive differential equations. The authors of \cite{feckan2} established results on the existence and uniqueness of $(\omega,{\mathbb T})$-periodic solutions for impulsive linear and semilinear problems. Furthermore, there are many papers on periodic solutions for periodic system on infinite-dimensional spaces (see \cite{aman}, \cite{liu2}, \cite{sata}, \cite{xiang1}); we also refer to \cite{guo} and \cite{park}, where the integro-differential systems on finite and infinite-dimensional Banach space are investigated. The existence of piecewise continuous mild solutions and optimal control of integro-differential systems is presented in \cite{wei}. In \cite{wang1}-\cite{wang11}, the integro-differential impulsive periodic systems on infinite-dimensional spaces are discussed. To our best knowledge, the existence and uniqueness of $(\omega, c)$-periodic solutions for integro-differential systems ($c\in{\mathbb C},\ c\neq 0)$ have not been extensively studied.

As a continuation of the investigations on $(\omega,{\mathbb T})$-periodic solutions for linear and semilinear problems, and periodic solutions for integro-differential impulsive periodic systems, we consider here $(\omega,\rho)$-periodic solutions of impulsive differential equations as a generalization of the previous concepts (see \cite{agaoglou}, \cite{aman}-\cite{bain1}, \cite{cooke}, \cite{feckan1}-\cite{feckan2}, \cite{nova-mono}-\cite{raff},  \cite{li1}-\cite{liu}, \cite{park}, \cite{sata}-\cite{xiang1}, \cite{wang}-\cite{wang4}). The main aim of this paper is to present results concerning the existence and uniqueness of $(\omega,\rho)$-periodic solutions for certain classes of abstract semilinear integro-differential impulsive equations, considered on a infinite-dimensional pivot Banach space $X$.

The organization of paper can be described briefly as follows. After recalling some preliminary results and definitions from the theory of $(\omega,c)$-periodic functions and strongly continuous semigroups of bounded operators, we present some results on the solutions of the nonhomogeneous linear impulsive equations and certain useful estimates for the further investigations. In the last section, we use the Banach fixed point theorem and the Schauder fixed point theorem to prove the existence and uniqueness of the $(\omega,\rho)$-solutions for semilinear integro-differential equations under our considerations.

\subsection{Preliminaries}  Let $I={\mathbb R}$ or $I=[0,\infty)$. By $(X,\|\cdot \|)$ is denoted a complex Banach space. The abbreviations ${\mathcal C}_{b}(I : X)$ and ${\mathcal C}(K:X),$ where $K$ is a non-empty compact subset of ${\mathbb R},$ stand for the spaces of bounded continuous functions $I \mapsto X$ and continuous functions $K\mapsto X,$ respectively. Both spaces are Banach ones endowed with the sup-norm. The space of $X$-valued piecewise continuous functions on $I$ is given by
\begin{align*}
{\mathcal PC}(I : X)& \equiv \bigl\{y : I\rightarrow X : y\in {\mathcal C}\bigl((t_i,t_{i+1}] : X\bigr),\\& t_i\neq0,\,\,\mbox{for all}\,\, i\in{\mathbb N},\,\, y(t_i^-)=y(t_i)\,\, \mbox{and}\,\, y(t_i^+) \mbox{ exist for any } i\in{\mathbb N}\bigr\},
\end{align*}
where the symbols $y(t_i^-)$ and $y(t_i^+)$ denote the left and the right limits of the function $y(t)$ at the point $t=t_i$,  $i\in{\mathbb N}$, respectively. Let us recall that $
{\mathcal PC}(I : X) $ is a Banach space endowed with the sup-norm.

For an operator family $(T(t))_{t\geq0}$ of a bounded linear operators on a Banach space $X$ it is said that is a strongly continuous semigroup of bounded linear operators (shortly $C_0$ semigroup) if and only if:
\begin{itemize}
\item[(i)] For all $t,\ s\geq 0,$ we have $T(t+s)=T(t)T(s)$;
\item[(ii)] $T(0)=E,$ the identity operator on $X$;
\item[(iii)] For every $x\in X$, we have $\lim\limits_{t\rightarrow0}T(t)x=x$.
\end{itemize}
Let $\rho: X\rightarrow X$. A function $f:[0,\infty)\rightarrow X$ is called $(\omega,\rho)$-periodic function (see \cite{feckan2}) if and only if there is a real number $\omega >0$ such that $f(t+\omega)=\rho f(t)$ for all $t\geq 0$. By ${\Phi}_{\omega,\rho}$ we denote the set of all piecewise continuous and $(\omega,\rho)$-periodic functions, i.e.,
\begin{align*}
{\Phi}_{\omega,\rho}=\bigl\{y\, :\, y\in{\mathcal PC}([0,\infty) : X)\,\, \mbox{and}\,\, y(\cdot+\omega)=\rho y(\cdot)\bigr\}.
\end{align*}

We continue the investigations started in \cite{feckan2} by studying the $(\omega,\rho)$-periodic solutions of the following abstract integro-differential impulsive equation
\begin{equation}\label{rav}
\left\{
\begin{aligned}
&\dot{y}(t)=Ay(t)+f\Biggl(t,y(t),\int\limits_{0}^tg(t,s,y(t))\, ds\Biggr), \quad t\neq{\tau}_k,\,\, k\in{\mathbb N};\\
&\Delta y|_{t={\tau}_k}=B_ky(t)+d_k,\quad \quad k\in{\mathbb N},\end{aligned}\right.
\end{equation}
where $A$ is the infinitesimal generator of a strongly continuous semigroup of bounded linear operators $(T(t))_{t\geq0}$, and $B_k$ is a
bounded linear operator on $X$ for all $k\in{\mathbb N}.$

\indent In this paper, we consider the following assumptions:
\begin{itemize}
\item[(A1)] $A$ is the infinitesimal generator of a strongly continuous semigroup of bounded operators $(T(t))_{t\geq0}$ in $X$. The operators $B_k$, $k\in{\mathbb N}$ are bounded linear operators and $T(t)B_k=B_kT(t)$, for all $k\in{\mathbb N}$, $t\geq0$.
\item[(A2)] The constants $d_k$ and the time sequence ${\tau}_k>0$ are such that $B_{k+m}=B_k$, $d_{k+m}=\rho d_k$, ${\tau}_{k+m}={\tau}_k+\omega$, $k\in{\mathbb N}$, for some fixed $m=i(0,\omega)$, where by $i(0,s)$ is denoted the number of impulsive points between $[0,s]$.
\item[(A3)] $\rho: X\rightarrow X$ is a linear isomorphism and $\rho A=A\rho$, $\rho B_k=B_k\rho$ for all $k\in{\mathbb N}$.
\item[(A4)] The operator $\rho-T(\omega)\prod_{k=1}^{m}(E+B_k)$ is injective.
\item[(A5)] For all $t\geq0$ and $y\in X,$ we have
\begin{align*}
 f\Biggl(t+\omega, \rho y,\rho\int\limits_0^{t+\omega}g(t,s,y)\, ds\Biggl)=\rho f\Biggl(t,y,\int\limits_0^tg(t,s,y)\, ds\Biggr).
\end{align*}
\item[(A6)] For all $t\geq s\geq0$ and $y\in X,$ it holds
\begin{align*}
g(t+\omega,s,\rho y)=\rho g(t,s,y).
\end{align*}
\item[(A7)] Let $f:[0,\infty)\times X\times X\rightarrow X$ and the function $t\mapsto (t,x,y)$ be measurable for all $(x,y)\in X\times X$. For every $\nu>0$, there exists $L_f(\nu)>0$ such that for almost all $t\geq 0$ and all $x_1,\ x_2, \ y_1,\ y_2\in X$ with $\|x_1\|,\ \|x_2\|,\ \|y_1\|,\ \|y_2\|\leq\nu,$ we have
\begin{align*}
\|f(t,x_1,y_1)-f(t,x_2,y_2)\|\leq L_f(\nu)\Biggl(\|x_1-x_2\|+\|y_1-y_2\|\Biggr).
\end{align*}
Set $D:=\{(t,s)\geq 0\times[0,\infty) : 0\leq s\leq t\}$. The function $g: D\times X\rightarrow X$ is continuous, and for each $\nu>0$ there exists $L_g(\nu)>0$ such that for each $(t,s)\in D$ and for each $x, y\in X$ with $\|x\|,\ \|y\|\leq\nu$, we have
\begin{align*}
\|g(t,s,x)-g(t,s,y)\|\leq L_g(\nu)\|x-y\|.
\end{align*}
\item[(A8)] There are constants $\alpha,\ \beta\geq0$ such that
\begin{align*}
\Biggl\|f\Biggl(t,y(t),\int\limits_0^tg(t,s,y(t))\, ds\Biggl)\Biggr\|\leq\alpha+\beta\|y\|,
\end{align*}
for any $t\geq 0$.
\item[(A9)] Let $M\geq1$ and $\gamma\in {\mathbb R}$ be such that $\|T(t)\|\leq Me^{\gamma t}$ for all $t\geq 0$.
\item[(A10)] $X$ is finite-dimensional.
\end{itemize}

\section{Nonhomogeneous linear impulsive problem}\label{homimpulse}

Of concern is the following equation
\begin{equation}\label{prva}
y'(t)=f\Biggl(t,y(t),\int\limits_0^tg(t,s,y(t))\, ds\Biggr)
\end{equation}
accompanied with the conditions
\begin{align}\label{usl1}
f\Biggl(t+\omega,\rho y, \rho\int\limits_{0}^{t+\omega}g(t,s,y)\, ds\Biggr)=\rho f\Biggl(t,y,\int\limits_0^t g(t,s,y)\, ds\Biggr)
\end{align}
and
\begin{align}\label{usl2}
g(t+\omega,s,\rho y)=\rho g(t,s,y).
\end{align}

We need the following auxiliary lemma:

\begin{lem}\label{lem1}
Let $f$ be continuous and locally Lipschitzian in last two coordinates. If
$$
f\Biggl(t+\omega, \rho y(t),\rho\int\limits_0^{t+\omega}g(t,s,y(t))\, ds\Biggr)=\rho f\Biggl(t,y(t),\int\limits_0^tg(t,s,y(t))\, ds\Biggr),\quad t\geq0,
$$
then $y(t)$ is a solution of equation
$$
y^{\prime}(t)=f\Biggl(t,y(t),\int_0^tg(t,s,y(t))\, ds\Biggr)
$$
satisfying $t\in{\Phi}_{\omega,\rho}$ if and only if we have
\begin{align*}
y(\omega)=\rho y(0).
\end{align*}
\end{lem}
\begin{proof}
Let $y\in{\Phi}_{\omega,\rho}$. By the definition of ${\Phi}_{\omega,\rho}$, we have $y(t+\omega)=\rho y(t)$ for all $t\geq 0$. Put $t=0$, so $y(\omega)=\rho y(0)$.
Now, let $y(\omega)=\rho y(0).$ Set
$x(t):={\rho}^{-1}y(t+\omega)$, $t\geq0$.
Using (\ref{usl1})-(\ref{usl2}), we have
\begin{align*}
x'(t)&={\rho}^{-1}y'(t+\omega)\\
&={\rho}^{-1}f\Biggl(t+\omega,y(t+\omega),\int\limits_0^{t+\omega}g(t+\omega,s,y(t+\omega))\, ds\Biggr)\\
&={\rho}^{-1}f\Biggl(t+\omega,\rho{\rho}^{-1} y(t+\omega),\int\limits_0^{t+\omega}g(t+\omega,s,\rho{\rho}^{-1} y(t+\omega))\, ds\Biggr)\\
&={\rho}^{-1}f\Biggl(t+\omega,\rho x(t),\int\limits_0^{t+\omega}g(t,s,\rho x(t))\, ds\Biggr)\end{align*}
\begin{align*}
&={\rho}^{-1}\rho f\Biggl(t+\omega,x(t),\int\limits_0^tg(t,s,y(t))\, ds\Biggr)\\
&=f\Biggl(t+\omega,x(t),\int\limits_0^tg(t,s,y(t))\, ds\Biggr).\end{align*}
Additionally,
\begin{align*}
x(0)={\rho}^{-1}y(\omega)=y(0).
\end{align*}
Now, both $y(t)$ and $x(t)$, $t\geq 0$, satisfy (\ref{prva}) and $y(0)=x(0)$. From the uniqueness of the solutions, we conclude that $x(t)=y(t)$, so $y(t+\omega)=\rho y(t)$, $t\geq 0$.
\end{proof}

At the very beginning of our work, we consider the homogeneous linear impulsive evolution equation.

\begin{lem}\label{lem2} Let \emph{(A1)--(A3)} hold. Then the homogeneous linear impulsive evolution equation
\begin{equation}\label{eq1}
\left\{
\begin{aligned}
&\dot{y}(t)=Ay(t), \quad t\neq{\tau}_k,\,\, k\in{\mathbb N}\\
&\Delta y|_{t={\tau}_k}=B_ky(t)+d_k,\quad\quad k\in{\mathbb N},\end{aligned}\right.
\end{equation}
has a solution $y\in{\Phi}_{\omega,\rho}$ if and only if $y(\omega)=\rho y(0)$ or
\begin{align*}
\Biggl(\rho-T(\omega)\Biggl(\prod\limits_{k=1}^{i(0,\omega)}(E+B_k){\Biggr)}\Biggr)y(0)=\sum\limits_{0<{\tau}_i<\omega}T(\omega-{\tau}_i)\prod\limits_{k=1}^{i({\tau}_i,\omega)}(E+B_k)d_i.
\end{align*}
\end{lem}
\begin{proof}
The direct implication is obvious. We prove the opposite direction. Let us assume that $y(\omega)=\rho y(0)$. The solution of (\ref{eq1}), with $y(0)$, for any $t\neq{\tau}_k$, $k\in{\mathbb N}$ is given by the formula (\cite{samoilenko}, (2.21)):
\begin{align*}
y(t)&=T(t)\prod\limits_{k=1}^{i(0,t)}(E+B_k)y(0)+\int\limits_0^t T(t-{\tau})\prod\limits_{k=1}^{i(\tau,t)}(E+B_k)f(\tau)\, d{\tau}\\
&+\sum\limits_{0<{\tau}_i<t}T(t-{\tau}_i)\prod\limits_{k=1}^{i({\tau}_i,t)}(E+B_k) d_i,\quad t\geq0.
\end{align*}
For every $t\geq0$ and $t\neq{\tau}_k$ for any $k\in{\mathbb N},$ we have:
\begin{align*}
y(t& +\omega)\\
& = T(t+\omega)\prod\limits_{k=1}^{i(0,t+\omega)}(E+B_k)y(0)+\sum\limits_{0<{\tau}_i<t+\omega}T(t+\omega-{\tau}_i)\prod\limits_{k=1}^{i({\tau}_i,t+\omega)}(E+B_k) d_i\\
&=T(t)T(\omega)\prod\limits_{k=1}^{i(0,t)}(E+B_k)\cdot\prod\limits_{k=1}^{i(0,\omega)}(E+B_k)y(0)\\
&+\sum\limits_{0<{\tau}_i<t+\omega}T(t)T(\omega-{\tau}_i)\prod\limits_{k=1}^{i(\omega,t+\omega)}(E+B_k)\cdot\prod\limits_{k=1}^{i({\tau}_i,t+\omega)}(E+B_k) d_i\end{align*}
\begin{align*}
&=T(t)\prod\limits_{k=1}^{i(0,t)}(E+B_k)\\
&\times{\Biggl(}T(\omega)\prod\limits_{k=1}^{i(0,\omega)}(E+B_k)y(0)+\sum\limits_{0<{\tau}_i<\omega}T(\omega-{\tau}_i)
\prod\limits_{k=1}^{i({\tau}_i,\omega)}(E+B_k) d_i{\Biggr)}\\
&+\sum\limits_{\omega\leq{\tau}_i<t+\omega}T(t+\omega-{\tau}_i)\prod\limits_{k=1}^{i({\tau}_i,t+\omega)} d_i\\
&=T(t)\prod\limits_{k=1}^{i(0,t)}y(\omega)+\sum\limits_{\omega\leq{\tau}_i<t+\omega}T(t+\omega-{\tau}_{i+m})\prod\limits_{k=1}^{i({\tau}_{i+m},t+\omega)} (E+B_k) d_{i+m}\\
&=T(t)\prod\limits_{k=1}^{i(0,t)}(E+B_k)\rho y(0)+\sum\limits_{0<{\tau}_i<t}T(t-{\tau}_i)\prod\limits_{k=1}^{i({\tau}_i,t)}\rho d_i=\rho y(t).
\end{align*}
\end{proof}
Next, we consider the $(\omega,\rho)$-periodic solutions of the following problem:

\begin{equation}\label{nonhom}
\left\{
\begin{aligned}
&\dot{y}(t)=Ay(t)+f(t), \quad t\neq{\tau}_k,\,\, k\in{\mathbb N}\\
&\Delta y|_{t={\tau}_k}=B_ky(t)+d_k,\quad\quad k\in{\mathbb N},\end{aligned}\right.
\end{equation}
where $f\in{\mathcal C}([0,\infty):X)$ and $f$ is an $(\omega,\rho)$-periodic function.

\begin{lem}\label{lem3} Let \emph{(A1)--(A4)} hold. Then the $(\omega,\rho)$-periodic solution $y\in \Psi={\mathcal PC}([0,\omega]:X)$ of (\ref{nonhom}) is given by
\begin{align*}
y(t)=\int\limits_0^{\omega}H(t,\tau) f(\tau)\, d\tau+\sum\limits_{i=1}^mH(t,{\tau}_i)d_i,
\end{align*}
where the function $H(\cdot,\cdot)$ is given by
\begin{equation*}
H(t,\tau)=
\left\{
\begin{aligned}
&{\Biggl(}T(t)\prod\limits_{k=1}^{i(0,t)}(E+B_k)\Biggl(\rho-T(\omega)\prod\limits_{k=1}^{i(0,\omega)}(E+B_k)\Biggr)^{-1}T(\omega-t)\prod\limits_{k=1}^{i(t,\omega)}(E+B_k)+E{\Biggr)}\\
&\times T(t-\tau)\prod\limits_{k=1}^{i(0,t)}(E+B_k),\quad 0<\tau<t;\\
&T(t)\prod\limits_{k=1}^{i(0,t)}(E+B_k)\Biggl(\rho-T(\omega)\prod\limits_{k=1}^{i(0,\omega)}(E+B_k)\Biggr)^{-1}T(\omega-\tau)\prod\limits_{k=1}^{i(t,\omega)}(E+B_k),\\
&t\leq\tau<\omega.
\end{aligned}\right.
\end{equation*}
\end{lem}
\begin{proof}
Using the formula \cite[(2.21)]{samoilenko} and Lemma \ref{lem1}, we obtain
\begin{align*}
y(\omega)&=T(\omega)\prod\limits_{k=1}^{i(0,\omega)}(E+B_k)y_0+\int\limits_0^{\omega}T(\omega-\tau)\prod\limits_{k=1}^{i(\tau,\omega)}(E+B_k)f(\tau)\, d\tau\\
&+\sum\limits_{i=1}^mT(\omega-{\tau}_i)\prod\limits_{k=1}^{i({\tau}_i,\omega)}(E+B_k) d_i=\rho y_0.
\end{align*}
Hence,
\begin{align*}
y_0&=\Biggl(\rho-T(\omega)\prod\limits_{k=1}^{i(0,\omega)}(E+B_k)\Biggr)^{-1}+\Biggl(\int\limits_0^{\omega}T(\omega-\tau)\prod\limits_{k=1}^{i(\tau,\omega)}(E+B_k)f(\tau) \, d\tau\\
&+\sum\limits_{i=1}^mT(\omega-{\tau}_i)\prod\limits_{k=1}^{i({\tau}_i,\omega)}(E+B_k) d_i\Biggr),
\end{align*}
so the solution of (\ref{nonhom}) can be written as\nopagebreak {\small
\begin{align*}
y(t)&=T(t)\prod\limits_{k=1}^{i(0,t)}(E+B_k)\Biggl(\rho-T(\omega)\prod\limits_{k=1}^{i(0,\omega)}(E+B_k)\Biggr)^{-1}\Biggl(\int\limits_{0}^{\omega}T(\omega-\tau) \prod\limits_{k=1}^{i(\tau,\omega)}(E+B_k)f(\tau)\, d\tau\\
&+\sum\limits_{i=1}^mT(\omega-{\tau}_i)\prod\limits_{k=1}^{i({\tau}_i,\omega)}(E+B_k) d_i\Biggr)+\int\limits_0^tT(t-\tau)\prod\limits_{k=1}^{i(\tau,t)}(E+B_k)f(\tau)\, d\tau\\
&+\sum\limits_{0<{\tau}_i<t}T(t-{\tau}_i)\prod\limits_{k=1}^{i({\tau}_i,t)}(E+B_k) d_i\\
&=\int\limits_0^{\omega}T(t)\prod\limits_{k=1}^{i(0.t)}(E+B_k)\Biggl(\rho-T(\omega)\prod\limits_{k=1}^{i(0,\omega)}(E+B_k)\Biggr)^{-1}T(\omega-\tau) \prod\limits_{k=1}^{i({\tau},\omega)}(E+B_k)f(\tau)\, d\tau\\
&+\sum\limits_{i=1}^mT(t)\prod\limits_{k=1}^{i(0,t)}\Biggl(\rho-T(\omega)\prod\limits_{k=1}^{i(0,\omega)}(E+B_k)\Biggr)^{-1}T(\omega-{\tau}_i) \prod\limits_{k=1}^{i({\tau}_i,\omega)}(E+B_k) d_i\\
&+\int\limits_0^tT(t-\tau)\prod\limits_{k=1}^{i(\tau,t)}(E+B_k)f(\tau)\, d\tau+\sum\limits_{0<{\tau}_i<t}T(t-{\tau}_i)\prod\limits_{k=1}^{i({\tau}_i,t)}(E+B_k) d_i\\
&=\int\limits_{0}^tT(t)\prod\limits_{k=1}^{i(0,t)}(E+B_k)\Biggl(\rho-T(\omega)\prod\limits_{k=1}^{i(0,\omega)}(E+B_k)\Biggr)^{-1}T(\omega-\tau) \prod\limits_{k=1}^{i(\tau,\omega)}(E+B_k)f(\tau)\, d\tau\\
&+\int\limits_t^{\omega}T(t)\prod\limits_{k=1}^{i(0,t)}(E+B_k)\Biggl(\rho-T(\omega)\prod\limits_{k=1}^{i(0,\omega)}(E+B_k)\Biggr)^{-1}T(\omega-\tau) \prod\limits_{k=1}^{i(\tau-\omega)}(E+B_k)f(\tau)\, d\tau\\
&+\sum\limits_{0<{\tau}_i<t}T(t)\prod\limits_{k=1}^{i(0,t)}(E+B_k)\Biggl(\rho-T(\omega)\prod\limits_{k=1}^{i(0,\omega)}(E+B_k)\Biggr)^{-1}T(\omega-{\tau}_i) \prod\limits_{k=1}^{i(\tau,\omega)}(E+B_k) d_i\\
&+\sum\limits_{t\leq{\tau}_i\leq\omega}T(t)\prod\limits_{k=1}^{i(0,t)}(E+B_k)\Biggl(\rho-T(\omega)\prod\limits_{k=1}^{i(0,\omega)}(E+B_k)\Biggr)^{-1}T(\omega-{\tau}_i) \prod\limits_{k=1}^{i(\tau,\omega)}(E+B_k) d_i\\
&+\int\limits_{0}^{t}T(t-\tau)\prod\limits_{k=1}^{i(\tau,t)}(E+B_k) f(\tau)\, d\tau
+\sum\limits_{0<{\tau}_i<t}T(t-{\tau}_i)\prod\limits_{k=1}^{i({\tau}_i,t)}(E+B_k) d_i\\
&=\int\limits_0^t\Biggl(T(t)\prod\limits_{k=1}^{i(0,t)}(E+B_k)\Biggl(\rho-T(\omega)\prod\limits_{k=1}^{i(0,\omega)}(E+B_k)\Biggr)^{-1}T(\omega-\tau) \prod\limits_{k=1}^{i(t,\omega)}(E+B_k)+E\Biggr)\\
\end{align*}
\begin{align*}
& \times T(t-\tau)\prod\limits_{k=1}^{i(\tau,t)}(E+B_k) f(\tau)\, d\tau\\
&+\int\limits_t^{\omega}T(t)\prod\limits_{k=1}^{i(0,t)}(E+B_k)\Biggl(\rho-T(\omega)\prod\limits_{k=1}^{i(0,\omega)}(E+B_k)\Biggr)^{-1}T(\omega-\tau) \prod\limits_{k=1}^{i(\tau,\omega)}(E+B_k)f(\tau)\, d\tau\\
&+\sum\limits_{0<{\tau}_i<t}\Biggl(T(t)\prod\limits_{k=1}^{i(0,t)}(E+B_k)\Biggl(\rho-T(\omega)\prod\limits_{k=1}^{i(0,\omega)}(E+B_k)\Biggr)^{-1}T(\omega-t) \prod\limits_{k=1}^{i(t,\omega)}(E+B_k)+E\Biggr)\\
&\times T(t-{\tau}_i)\prod\limits_{k=1}^{i({\tau}_i,t)}(E+B_k) d_i\\
&+\sum\limits_{t\leq{\tau}_i<\omega} T(t)\prod\limits_{k=1}^{i(0,t)}(E+B_k)\Biggl(\rho-T(\omega)\prod\limits_{k=1}^{i(0,\omega)}(E+B_k)\Biggr)^{-1} T(\omega-{\tau}_i)\prod\limits_{k=1}^{i({\tau}_i,\omega)} d_i\\
&=\int\limits_0^{\omega}H(t,\tau) f(\tau)\, d\tau+\sum\limits_{i=1}^m H(t,{\tau}_i) d_i.
\end{align*}}
\end{proof}
As a consequence of the previous result, we can state the following:

\begin{lem} Let \emph{(A1)--(A4)} hold, and let for each $t\geq 0$ we have
$$
\Biggl(\rho-T(\omega)\prod\limits_{k=1}^{i(0,\omega)}(E+B_k)\Biggr)^{-1}T(t)=T(t)\Biggl(\rho-T(\omega)\prod\limits_{k=1}^{i(0,\omega)}(E+B_k)\Biggr)^{-1}.
$$
Then the unique $(\omega,\rho)$-periodic solution $y\in{\Psi}$  of (\eqref{nonhom}) is given by
\begin{align*}
y(t)=\int\limits_{0}^{\omega}H(t,\tau) f(\tau)\, d\tau+\sum\limits_{i=1}^m H(t,{\tau}_i) d_i,
\end{align*}
where $H(\cdot,\cdot)$ is defined by{\small
\begin{equation*}
H(t,\tau):=
\left\{
\begin{aligned}
&\rho\Biggl(\rho-T(\omega)\prod\limits_{k=1}^{i(0,\omega)}(E+B_k)\Biggr)^{-1}T(t-\tau)\prod\limits_{k=1}^{i(\tau,t)}(E+B_k),\quad 0<\tau<t;\\
&T(t+\omega-\tau)\prod\limits_{k=1}^{i(0,t)+i(\tau,\omega)}(E+B_k)\Biggl(\rho-T(\omega)\prod\limits_{k=1}^{i(0,\omega)}(E+B_k)\Biggr)^{-1},\quad t\leq\tau<\omega.\end{aligned}\right.
\end{equation*}}
\end{lem}

\begin{proof}
By the foregoing, the unique solution of considered problem is given by
\begin{align*}
y(t)&=T(t)\prod\limits_{k=1}^{i(0,t)}(E+B_k)y_0+\int\limits_0^t T(t-\tau)\prod\limits_{k=1}^{i(\tau,t)}(E+B_k) f(\tau)\, d\tau\\
&+\sum\limits_{0<{\tau}_i<t}T(t-{\tau}_i)\prod\limits_{k=1}^{i({\tau}_i,t)} d_i.
\end{align*}
Using Lemma \ref{lem1}, we obtain
\begin{align*}
y(\omega)&=T(\omega)\prod\limits_{k=1}^{i(0,\omega)}(E+B_k)y_0+\int\limits_0^{\omega}T(\omega-\tau)\prod\limits_{k=1}^{i(\tau,\omega)}(E+B_k) f(\tau)\, d\tau\\
&+\sum\limits_{i=1}^{m}T(\omega-{\tau}_i)\prod\limits_{k=1}^{i({\tau}_i,\omega)} d_i=\rho y_0.
\end{align*}
Now,
\begin{align*}
y_0&=\Biggl(\rho-T(\omega)\prod\limits_{k=1}^{i(0,\omega)}(E+B_k)\Biggr)^{-1}\Biggl(T(\omega-\tau)\prod\limits_{k=1}^{i(\tau,\omega)}(E+B_k) f(\tau)\, d\tau\\
&+\sum\limits_{i=1}^m T(\omega-{\tau}_i)\prod\limits_{k=1}^{i({\tau}_i,\omega)} d_i\Biggr).
\end{align*}
Hence, the solution of (\ref{nonhom}) can be written as
\begin{align*}
y(t)&=\\
&=T(t)\prod\limits_{k=1}^{i(0,t)}(E+B_k)\Biggl(\rho-T(\omega)\prod\limits_{k=1}^{i(0,\omega)}(E+B_k)\Biggr)^{-1}\Biggl(\int\limits_0^{\omega} T(\omega-\tau)\prod\limits_{k=1}^{i(\tau,\omega)}(E+B_k)f(\tau)\, d\tau\\
&+\sum\limits_{i=1}^{m}T(\omega-{\tau}_i)\prod\limits_{k=1}^{i({\tau}_i,\omega)}(E+B_k) d_i\Biggr)+\int\limits_0^t T(t-\tau)\prod\limits_{k=1}^{i(\tau,t)}(E+B_k)f(\tau)\, d\tau\\
&+\sum\limits_{0<{\tau}_i<t} T(t-{\tau}_i)\prod\limits_{k=1}^{i({\tau}_i,t)}(E+B_k) d_i\\
&=\int\limits_0^t T(t)\prod\limits_{k=1}^{i(0,t)}(E+B_k)\Biggl(\rho-T(\omega)\prod\limits_{k=1}^{i(0,\omega)}(E+B_k)\Biggr)^{-1} T(\omega-\tau) \prod\limits_{k=1}^{i(\tau,\omega)}(E+B_k) f(\tau)\, d\tau\\
&+\int\limits_{t}^{\omega} T(t)\prod\limits_{k=1}^{i(0,t)}(E+B_k) T(\omega-\tau) \prod\limits_{k=1}^{i(\tau,\omega)}(E+B_k) f(\tau)\, d\tau\\
&+\sum\limits_{0<{\tau}_i<t}T(t)\prod\limits_{k=1}^{i(0,t)}(E+B_k)\Biggl(\rho-T(\omega)\prod\limits_{k=1}^{i(0,\omega)}(E+B_k)\Biggr)^{-1} T(\omega-{\tau}_i) \prod\limits_{k=1}^{i({\tau}_i,\omega)}(E+B_k) f(\tau)\, d\tau\\
&+\sum\limits_{t\leq{\tau}_i<\omega} T(t)\prod\limits_{k=1}^{i(0,t)}(E+B_k)\Biggl(\rho-T(\omega)\prod\limits_{k=1}^{i(0,\omega)}(E+B_k)\Biggr)^{-1} T(\omega-{\tau}_i) \prod\limits_{k=1}^{i({\tau}_i,\omega)}(E+B_k) d_i
\end{align*}
\begin{align*}
&+\int\limits_0^t T(t-\tau)\prod\limits_{k=1}^{i(\tau,t)}(E+B_k) f(\tau)\, d\tau+\sum\limits_{0<{\tau}_i<t} T(t-{\tau}_i)\prod\limits_{k=1}^{i({\tau}_i,t)}(E+B_k) d_i\\
&=\int\limits_0^t \rho\Biggl(\rho-T(\omega)\prod\limits_{k=1}^{i(0,\omega)}(E+B_k)\Biggr)^{-1}T(t-\tau)\prod\limits_{k=1}^{i(\tau,t)}(E+B_k) f(\tau)\, d\tau\\
&+\int\limits_t^{\omega} T(t+\omega-\tau)\prod\limits_{k=1}^{i(0,t)+i(\tau,\omega)}(E+B_k) \Biggl(\rho-T(\omega)\prod\limits_{k=1}^{i(0,\omega)}(E+B_k)\Biggr)^{-1} f(\tau)\, d\tau\\
&+\sum\limits_{0<{\tau}_i<t}\rho\Biggl(\rho-T(\omega)\prod\limits_{k=1}^{i(0,\omega)}(E+B_k)\Biggr)^{-1} T(t-{\tau}_i)\prod\limits_{k=1}^{i({\tau}_i,t)}(E+B_k) d_i\end{align*}
\begin{align*}
&+\sum\limits_{t\leq{\tau}_i<\omega} T(t+\omega-{\tau}_i)\prod\limits_{k=1}^{i(0,t)+i({\tau}_i,\omega)}(E+B_k)\Biggl(\rho-T(\omega)\prod\limits_{k=1}^{i(0,\omega)}(E+B_k)\Biggr)^{-1} d_i\\
&=\int\limits_0^{\omega} H(t,{\tau}) f(\tau)\, d\tau+\sum\limits_{i=1}^m H(t,{\tau}_i) d_i.
\end{align*}
\end{proof}
Next, we give the following estimates on $\sum_{i=1}^m\|H(t,\tau) d_i\|$ and $\int_0^{\omega}\|H(t,\tau)\|\, d\tau$:

\begin{lem}\label{est1} Let \emph{(A1)--(A4)} hold. Then
$$\sum\limits_{i=1}^m\|H(t,{\tau}_i)\|d_i\leq C_1$$ {\small
\begin{equation*}
\equiv\left\{
\begin{aligned}
&M\max\Biggl\{\Bigl\|\prod\limits_{k=1}^{m}(E+B_k)^2\Bigr\|,1\Biggr\}\cdot\max\bigl\{e^{2\gamma\omega},1\bigr\}\cdot
\Biggl(M\Biggl\|\Biggl(\rho-T(\omega)\prod\limits_{k=1}^{i(0,\omega)}(E+B_k)\Biggr)^{-1}\Biggr\|+1\Biggr)\\
&\times\sum\limits_{1<i<m}e^{\gamma(\omega-{\tau}_i)}\|d_i\|,\quad \gamma>0;\\
&M\max\Biggl\{\Bigl\|\prod\limits_{k=1}^m(E+B_k)^2\Bigr\|,1\Biggr\}\Biggl(M\Biggl\|\Biggl(\rho-T(\omega)\prod\limits_{k=1}^{i(0,\omega)}(E+B_k)\Biggr)^{-1}\Biggr\|+1\Biggr)\\
&\times\sum\limits_{1<i<m}\|d_i\|,\quad \gamma\leq0,\end{aligned}\right.
\end{equation*}}
for any $t\in[0,\omega]$.
\end{lem}

\begin{proof}
We have
\begin{align*}
\sum\limits_{i=1}^m\|H(t,{\tau}_i)\|\cdot\|d_i\|
=\sum\limits_{0<{\tau}_i<t}\|H(t,{\tau}_i)\|\cdot\|d_i\|+\sum\limits_{t\leq{\tau}_i<\omega}\|H(t,{\tau}_i)\|\cdot\|d_i\| ,
\end{align*}
so that
\begin{align*}
&\sum\limits_{i=1}^m\|H(t,{\tau}_i)\|\cdot\|d_i\|\\
&\leq\sum\limits_{0<{\tau}_i<t}\Biggl(\|T(t)\|\biggl\|\prod\limits_{k=1}^{i(0,t)}(E+B_k)\Biggr\|\cdot \Biggl\|\Biggl(\rho-T(\omega)\prod\limits_{k=1}^{i(0,\omega)}(E+B_k)\Biggr)^{-1}\Biggr\| \cdot\|T(\omega-{\tau}_i)\|\\
&\times \Biggl\|\prod\limits_{k=1}^{i({\tau}_i,\omega)}(E+B_k)\Biggr\|+\|T(t-{\tau}_i)\|\cdot \Biggl\|\prod\limits_{k=1}^{i({\tau}_i,t)}(E+B_k)\Biggr\|\Biggr)\|d_i\|+ \sum\limits_{t\leq{\tau}_i<\omega}\|T(t)\|\\
&\times\|\prod\limits_{k=1}^{i(0,t)}(E+B_k)\|\cdot \Biggl\|\Biggl(\rho-T(\omega)\prod\limits_{k=1}^{i(0,\omega)}(E+B_k)\Biggr)^{-1}\Biggr\|\cdot\|T(\omega-{\tau}_i)\|\cdot \Biggl\|\prod\limits_{k=1}^{i({\tau}_i,\omega)}(E+B_k)\Biggr\|\cdot\|d_i\| \\
&\leq\sum\limits_{0<{\tau}_i<\omega}\|T(t)\|\cdot \Biggl\|\prod\limits_{k=1}^{i(0,t)}(E+B_k)\Biggr\|\cdot \Biggl\|\Biggl(\rho-T(\omega)\prod\limits_{k=1}^{i(0,\omega)}(E+B_k)\Biggr)^{-1}\Biggr\|
\cdot\|T(\omega-{\tau}_i)\|\\
&\times \Biggl\|\prod\limits_{k=1}^{i({\tau}_i,\omega)}(E+B_k)\cdot \Biggr\| \|d_i\|+\sum\limits_{0<{\tau}_i<t}\|T(t-{\tau}_i)\|\cdot \Biggl\|\prod\limits_{k=1}^{i({\tau}_i,t)}(E+B_k)
\Biggr\|\cdot\|d_i\|\\
&\leq \Biggl\|\Biggl(\rho-T(\omega)\prod\limits_{k=1}^{i(0,\omega)}(E+B_k)\Biggr)^{-1}\Biggr\|\sum\limits_{0<{\tau}_i<\omega}Me^{\gamma t}\Biggl \|\prod\limits_{k=1}^{i(0,t)}(E+B_k)\Biggr\|Me^{\gamma(\omega-{\tau}_i)}\\
&\times \Biggl\|\prod\limits_{k=1}^{i({\tau}_i,\omega)}(E+B_k)\Biggr\|\cdot\|d_i\|+\sum\limits_{0<{\tau}_i<t}Me^{\gamma(t-{\tau}_i)} \Biggl\|\prod\limits_{k=1}^{i({\tau}_i,t)}(E+B_k)\Biggr\| \cdot\|d_i\|\\
&\leq \Biggl\|\Biggl(\rho-T(\omega)\prod\limits_{k=1}^{i(0,\omega)}(E+B_k)\Biggr)^{-1}\Biggr\|\sum\limits_{0<{\tau}_i<\omega}M^2e^{\gamma(\omega+t-{\tau}-i)} \Biggl\|\prod\limits_{k=1}^{i(0,t)}(E+B_k)\Biggr\|\\
&\times \Biggl\|\prod\limits_{k=1}^{i({\tau}_i,\omega)}(E+B_k)\Biggr\|\cdot\|d_i\|
+\sum\limits_{0<{\tau}_i<t}Me^{\gamma(t-{\tau}_i)}\Biggl \|\prod\limits_{k=1}^{i({\tau}_i,t)}(E+B_k)\Biggr\|\cdot\|d_i\|\\
&\leq M\max\Biggl\{\prod\limits_{k=1}^{i(0,\omega)}(E+B_k)^2\|,1\Biggr\}\cdot\Biggl(\Biggl\|\Biggl(\rho-T(\omega)\prod\limits_{k=1}^{i(0,\omega)}(E+B_k)\Biggr)^{-1}\Biggr\|\\ &\times\sum\limits_{0<{\tau}_i<\omega}Me^{\gamma(\omega+t-{\tau}_i)}\|d_i\|+\sum\limits_{0<{\tau}_i<t}e^{\gamma(t-{\tau}_i)}\|d_i\|\Biggr),
\end{align*}
for any $t\in[0,\omega]$. We will consider separately the following two cases: $\gamma>0$ and $\gamma\leq 0$.
For $\gamma>0$, we have
\begin{align*}
&\sum\limits_{i=1}^m\|H(t,{\tau}_i)\|\cdot\|d_i\|\\
&\leq M\max\Biggl\{\Biggl\|\prod\limits_{k=1}^{i(0,\omega)}(E+B_k)^2\Biggr\|,1\Biggr\}\cdot\Biggl(\Biggl\|\Biggl(\rho-T(\omega)\prod\limits_{k=1}^{i(0,\omega)} (E+B_k)\Biggr)^{-1}\Biggr\|\end{align*}
\begin{align*}
&\times\sum\limits_{0<{\tau}_i<\omega}Me^{\gamma(\omega+t-{\tau}_i)}\|d_i\|+\sum\limits_{0<{\tau}_i<t}e^{\gamma(t-{\tau}_i)}\|d_i\|\Biggr)\\
&\leq M\max\Biggl\{\Biggl\|\prod\limits_{k=1}^{i(0,\omega)}(E+B_k)^2\Biggr\|,1\Biggr\}\max\{e^{\gamma\omega},1\}\cdot \Biggl(\Biggl\|\Biggl(\rho-T(\omega)\prod\limits_{k=1}^{i(0,\omega)}(E+B_k)\Biggr)^{-1}\Biggr\|\\
&\times\sum\limits_{0<{\tau}_i<\omega}Me^{\gamma(t-{\tau}_i)}\|d_i\|+\sum\limits_{0<{\tau}_i<t}e^{\gamma(t-{\tau}_i)}\|d_i\|\Biggr)\\
&\leq M\max\Biggl\{\Biggl\|\prod\limits_{k=1}^{i(0,\omega)}(E+B_k)^2\Biggr\|,1\Biggr\}\max \bigl\{e^{\gamma\omega},1\bigr\}\\
&\times\Biggl(M\Biggl\|\Biggl(\rho-T(\omega)\prod\limits_{k=1}^{i(0,\omega)}(E+B_k)\Biggr)^{-1}\Biggr\|+1\Biggr)\cdot\sum\limits_{1<i<m}e^{\gamma(\omega-{\tau}_i)}\|d_i\|.
\end{align*}
For $\gamma\leq0$, we have
\begin{align*}
&\sum\limits_{i=1}^m\|H(t,{\tau}_i)\|\cdot\|d_i\|\\
&\leq M\max\Biggl\{\Biggl\|\prod\limits_{k=1}^{i(0,\omega)}(E+B_k)^2\Biggr\|,1\Biggr\}\cdot\Biggl(\Biggl\|\Biggl(\rho-T(\omega)\prod\limits_{k=1}^{i(0,\omega)}(E+B_k)\Biggr)^{-1}\Biggr\| \end{align*}
\begin{align*}
&\times\sum\limits_{0<{\tau}_i<\omega}Me^{\gamma(\omega+t-{\tau}_i)}\|d_i\|+\sum\limits_{0<{\tau}_i<t}e^{\gamma(t-{\tau}_i)}\|d_i\|\Biggr)\\
&\leq M\max\Biggl\{\Biggl\|\prod\limits_{k=1}^{i(0,t)}(E+B_k)^2\Biggr\|,1\Biggr\}\cdot\Biggl(M\Biggl\|\Biggl(\rho-T(\omega)\prod\limits_{k=1}^{i(0,\omega)}(E+B_k)\Biggr)^{-1}\Biggr\|+1\Biggr)\\ &\times\sum\limits_{1<i<m}\|d_i\|.
\end{align*}
\end{proof}
Under an additional condition, we can state the following modified version of previous lemma:

\begin{lem}\label{est11}
Let \emph{(A1)--(A4)} hold, and let
$$
\Biggl(\rho-T(\omega)\prod\limits_{k=1}^{i(0,\omega)}(E+B_k)\Biggr)^{-1}T(t)=T(t)\Biggl(\rho-T(\omega)\prod\limits_{k=1}^{i(0,\omega)}(E+B_k)\Biggr)^{-1},\quad t\geq 0.
$$
Then
\begin{equation*}
\sum\limits_{i=1}^{m}\|H(t,{\tau}_i)d_i\|\\
\leq  C'_1\equiv
\left\{
\begin{aligned}
&M\Biggl\|\Biggl(\rho-T(\omega)\prod\limits_{k=1}^{i(0,\omega)}(E+B_k)\Biggr)^{-1}\Biggr\|\max\Biggl\{\Biggl\|\prod\limits_{k=1}^{i(0,\omega)}(E+B_k)\Biggr\|,1\Biggr\}\\
&\times\max\bigl\{\|T\|,e^{\gamma\omega}\bigr\}\cdot\sum\limits_{0<{\tau}_i<\omega}e^{\gamma(\omega-{\tau}_i)}\|d_i\|,\quad\gamma>0;\\
&M\Biggl\|\Biggl(\rho-T(\omega)\prod\limits_{k=1}^{i(0,\omega)}(E+B_k)\Biggr)^{-1}\Biggr\|\max\Biggl\{\Biggl\|\prod\limits_{k=1}^{i(0,\omega)}(E+B_k)\Biggr\|,1\Biggr\}\\
&\times\max\{\|T\||,1\}\cdot\sum\limits_{1<i<m}\|d_i\|,\quad \gamma\leq0.
\end{aligned}\right.
\end{equation*}
\end{lem}

\begin{proof}
Using (A9), we obtain
\begin{align*}
&\sum\limits_{i=1}^m\|H(t,{\tau}_i)\|\cdot\|d_i\|=\sum\limits_{0<{\tau}_i<t}\|H(t,{\tau}_i)\|\cdot\|d_i\|+\sum\limits_{t\leq{\tau}_i<\omega} \|H(t,{\tau}_i)\|\cdot\|d_i\|\\
&\leq \sum\limits_{0<{\tau}_i<t}\|T\|\cdot \Biggl\|\Biggl(\rho-T(\omega)\prod\limits_{k=1}^{i(0,\omega)}(E+B_k)\Biggr)^{-1}\Biggr\| \cdot \bigl\|T(t-{\tau}_i)\bigr\| \cdot \Biggl\|\prod\limits_{k=1}^{i({\tau}_i,t)}(E+B_k)\Biggr\|\\
&+\sum\limits_{t\leq{\tau}_i<\omega}\bigl\|T(t+\omega-{\tau}_i)\bigr\|\cdot \Biggl\|\prod\limits_{k=1}^{i(0,t)+i({\tau}_i,\omega)}(E+B_k)\Biggr\| \cdot \Biggl\|\Biggl(\rho-T(\omega)\prod\limits_{k=1}^{i(0,\omega)}(E+B_k)\Biggr)^{-1}\Biggr\|\cdot\|d_i\| \\
&\leq M \Biggl\|\Biggl(\rho-T(\omega)\prod\limits_{k=1}^{i(0,\omega)}(E+B_k)\Biggr)^{-1}\Biggr\| \max\Biggl\{\Biggl\|\prod\limits_{k=1}^{i(0,\omega)}(E+B_k)\Biggl\|,1\Biggr\}\\
&\times\Biggl(\sum\limits_{0<{\tau}_i<t}\|T\|e^{\gamma(t-{\tau}_i)}\|d_i\|+\sum\limits_{t\leq{\tau}_i\leq\omega} e^{\gamma(t+\omega-{\tau}_i)}\|d_i\|\Biggr),
\end{align*}
for all $t\in[0,\omega]$. The following two cases emerge: $\gamma>0$ and $\gamma\leq0$.
For $\gamma>0$, we have
\begin{align*}
&\sum\limits_{i=1}^m\|H(t,{\tau}_i)\|\cdot\|d_i\| \\
&\leq M\Biggl\|\Biggl(\rho-T(\omega)\prod\limits_{k=1}^{i(0,\omega)}(E+B_k)\Biggr)^{-1}\Biggr\|\max\Biggl\{\Biggl\|\prod\limits_{k=1}^{i(0,\omega)}(E+B_k)\Biggr\|,1\Biggr\}\\
&\times\Biggl(\sum\limits_{0<{\tau}_i<t}\|T\| e^{\gamma(\omega-{\tau}_i)}\|d_i\|+\sum\limits_{t\leq{\tau}_i\leq\omega}e^{\gamma(2\omega-{\tau}_i)}\|d_i\|\Biggr)\\
&\leq M\Biggl\|\Biggl(\rho-T(\omega)\prod\limits_{k=1}^{i(0,\omega)}(E+B_k)\Biggr)^{-1}\Biggr\|\max\Biggl\{\Biggl\|\prod\limits_{k=1}^{i(0,\omega)}(E+B_k)\Biggr\|,1\Biggr\}\\
&\times\max \bigl\{\|T\|,e^{\gamma\omega}\bigr\}\cdot\sum\limits_{0<{\tau}_i<\omega}e^{\gamma(\omega-{\tau}_i)}\|d_i\|.
\end{align*}
For $\gamma\leq0$, we have
\begin{align*}
&\sum\limits_{i=1}^m\|H(t,{\tau}_i)\|\cdot\|d_i\|\\
&\leq M\Biggl\|\Biggl(\rho-T(\omega)\prod\limits_{k=1}^{i(0,\omega)}(E+B_k)\Biggr)^{-1}\Biggr\|\max\Biggl\{\Biggl\|\prod\limits_{k=1}^{i(0,\omega)}(E+B_k)\Biggr\|,1\Biggr\}\\
&\times\Biggl(\sum\limits_{0<{\tau}_i<t}\|T\|\cdot\|d_i\|+\sum\limits_{t\leq{\tau}_i\leq\omega}\|d_i\|\Biggr)\\
&\leq M\Biggl\|\Biggl(\rho-T(\omega)\prod\limits_{k=1}^{i(0,\omega)}(E+B_k)\Biggr)^{-1}\Biggr\|\max\Biggl\{\Biggl\|\prod\limits_{k=1}^{i(0,\omega)}(E+B_k)\Biggr\|,1\Biggr\}\\
&\times\max\{\|T\|,1\}\cdot\sum\limits_{1<i<m}\|d_i\|.
\end{align*}
This completes the proof of lemma.
\end{proof}

We also need the following lemma:

\begin{lem}\label{est2} Let \emph{(A1)--(A4)} hold, and let $t\in(0,\omega)$. Then
\begin{equation*}
\int\limits_0^{\omega}\|H(t,{\tau})\|\, d\tau\leq C_2\end{equation*}
\begin{equation*}
\equiv \left\{
\begin{aligned}
&\Biggl\{M\max\Biggl\{\Biggl\|\prod\limits_{k=1}^{i(0,\omega)}(E+B_k)^2\Biggr\|,1\Biggr\}\Biggl\|\Biggl(\rho-T(\omega)\prod\limits_{k=1}^{i(0,\omega)}(E+B_k)\Biggr)^{-1}\Biggr\| e^{\gamma\omega}\\
&+\max\Biggl\{\Biggl\|\prod\limits_{k=1}^{i(0,\omega)}(E+B_k)\Biggr\|,1\Biggr\}\Biggr\}M\frac{e^{\gamma\omega}-1}{\gamma},\quad \gamma\neq0;\\
&\Biggl\{M\max\Biggl\{\Biggl\|\prod\limits_{k=1}^{i(0,\omega)}(E+B_k)^2\Biggr\|,1\Biggr\}\Biggl\|\Biggl(\rho-T(\omega)\prod\limits_{k=1}^{i(0,\omega)}(E+B_k)\Biggr)^{-1}\Biggr\|\\
&+\max\Biggl\{\Biggl\|\prod\limits_{k=1}^{i(0,\omega)}(E+B_k)\Biggr\|,1\Biggr\}M\omega,\quad \gamma=0.
\end{aligned}\right.
\end{equation*}
\end{lem}

\begin{proof}
We have
\begin{align*}
&\int\limits_{0}^{\omega}\|H(t,\tau)\|\,d\tau\\
&=\int\limits_0^t\Biggl\|\Biggl(T(t)\prod\limits_{k=1}^{i(0,t)}(E+B_k)\Biggl(\rho-T(\omega)\prod\limits_{k=1}^{i(0,\omega)}(E+B_k)\Biggr)^{-1}T(\omega-t) \prod\limits_{k=1}^{i(t,\omega)}(E+B_k)+E\Biggr)\\
&\times T(t-\tau)\prod\limits_{k=1}^{i(\tau,t)}(E+B_k)\Biggr\|\, d\tau\\
&+\int\limits_{t}^{\omega}\Biggl\|T(t)\prod\limits_{k=1}^{i(0,t)}(E+B_k)\Biggl(\rho-T(\omega)\prod\limits_{k=1}^{i(0,\omega)}(E+B_k)\Biggr)^{-1} T(\omega-\tau) \prod\limits_{k=1}^{i(\tau,\omega)}(E+B_k)\Biggr\|\, d\tau\\
&\leq \int\limits_0^t \|T(t)\|\cdot \Biggl\|\prod\limits_{k=1}^{i(0,t)}(E+B_k)\Biggr\|\cdot \Biggl\|\Biggl(\rho-T(\omega)\prod\limits_{k=1}^{i(0,\omega)}(E+B_k)\Biggr)^{-1}\Biggr\| \cdot\|T(\omega-\tau)\|\\
&\times \Biggl\|\prod\limits_{k=1}^{i(\tau,\omega)}(E+B_k)\Biggr\|\, d\tau+\int\limits_{0}^t\|T(t-\tau)\|\cdot \Biggl\|\prod\limits_{k=1}^{i(\tau,t)}(E+B_k)\Biggr\|\, d\tau\\
&+\int\limits_t^{\omega}\|T(t)\|\cdot \Biggl\|\prod\limits_{k=1}^{i(0,t)}(E+B_k)\Biggr\|\cdot \Biggl\|\Biggl(\rho-T(\omega)\prod\limits_{k=1}^{i(0,\omega)}(E+B_k)\Biggr)^{-1}\Biggr\| \\
&\times\|T(\omega-\tau)\|\cdot \Biggl\|\prod\limits_{k=1}^{i(\tau,\omega)}(E+B_k)\Biggr\|\, d\tau\\
&\leq M^2\max\Biggl\{\Biggl\|\prod\limits_{k=1}^{i(0,\omega)}(E+B_k)^2\Biggr\|,1\Biggr\}\Biggl\|\Biggl(\rho-T(\omega)\prod\limits_{k=1}^{i(0,\omega)}(E+B_k)\Biggr)^{-1}\Biggr\| \cdot\int\limits_{0}^{\omega}e^{\gamma(\omega+t-\tau)}\, d\tau\\
&+M\max\Biggl\{\Biggl\|\prod\limits_{k=1}^{i(0,\omega)}(E+B_k)\Biggr\|,1\Biggr\}\cdot\int\limits_{0}^te^{\gamma(t-\tau)}\, d\tau.
\end{align*}
Now, we have two subcases: $\gamma\neq0$ and $\gamma=0$.
For $\gamma\neq0$, we have
\begin{align*}
&\int\limits_{0}^{\omega}\|H(t,{\tau})\|\, d\tau\\
&\leq M^2\max\Biggl\{\Biggl\|\prod\limits_{k=1}^{i(0,\omega)}(E+B_k)^2\Biggr\|,1\Biggr\}\cdot \Biggl\|\Biggl(\rho-T(\omega)\prod\limits_{k=1}^{i(0,\omega)}(E+B_k)\Biggr)^{-1}\Biggr\| \cdot\int\limits_{0}^{\omega}e^{\gamma(\omega+t-\tau)}\,d\tau\\
&+M\max\Biggl\{\|\prod\limits_{k=1}^{i(0,\omega)}(E+B_k)\|,1\Biggr\}\cdot\int\limits_0^te^{\gamma(t-\tau)}\, d\tau\end{align*}
\begin{align*}
&\leq M^2\max\Biggl\{\Biggl\|\prod\limits_{k=1}^{i(0,\omega)}(E+B_k)^2\Biggr\|,1\Biggr\}\cdot \Biggl\|\Biggl(\rho-T(\omega)\prod\limits_{k=1}^{i(0,\omega)}(E+B_k)\Biggr)^{-1}\Biggr\| \frac{e^{\gamma(t+\omega)}-e^{\gamma t}}{\gamma}\\
&+M\max\Biggl\{\Biggl\|\prod\limits_{k=1}^{i(0,\omega)}(E+B_k)\Biggr\|,1\Biggr\}\cdot\frac{e^{\gamma t}-1}{\gamma}\\
&\leq\Biggl\{ M\max\Biggl\{\Biggl\|\prod\limits_{k=1}^{i(0,\omega)}(E+B_k)^2\Biggr\|,1\Biggr\}\cdot \Biggl\|\Biggl(\rho-T(\omega)\prod\limits_{k=1}^{i(0,\omega)}(E+B_k)\Biggr)^{-1}\Biggr\| e^{\gamma\omega}\\
&+\max\Biggl\{\Biggl\|\prod\limits_{k=1}^{i(0,\omega)}(E+B_k)\Biggr\|,1\Biggr\}\Biggr\}M\frac{e^{\gamma\omega}-1}{\gamma}.
\end{align*}
For $\gamma=0$, we have
\begin{align*}
&\int\limits_{0}^{\omega}\|H(t,\tau)\|\, d\tau\\
&\leq M^2\max\Biggl\{\Biggl\|\prod\limits_{k=1}^{i(0,\omega)}(E+B_k)^2\Biggr\|,1\Biggr\}\cdot \Biggl\|\Biggl(\rho-T(\omega)\prod\limits_{k=1}^{i(0,\omega)}(E+B_k)\Biggr)^{-1}\Biggr\| \cdot\int\limits_{0}^{\omega}e^{\gamma(\omega+t-\tau)}\, d\tau\\
&+M\max\Biggl\{\Biggl\|\prod\limits_{k=1}^{i(0,\omega)}(E+B_k)\Biggr\|,1\Biggr\}\cdot\int\limits_0^te^{\gamma(t-\tau)}\, d\tau\\
&\leq\Biggl\{M\max\Biggl\{\Biggl\|\prod\limits_{k=1}^{i(0,\omega)}(E+B_k)^2\Biggr\|,1\Biggr\}\cdot \Biggl\|\Biggl(\rho-T(\omega)\prod\limits_{k=1}^{i(0,\omega)}(E+B_k)\Biggr)^{-1}\Biggr\| \\
&+\max\Biggl\{\Biggl\|\prod\limits_{k=1}^{i(0,\omega)}(E+B_k)\Biggr\|,1\Biggr\}\Biggr\}M\omega.
\end{align*}
\end{proof}
With an additional condition from the previous lemma, we can state the following:

\begin{lem}\label{est22} Let \emph{(A1)--(A4)} hold, and let $$
\Biggl(\rho-T(\omega)\prod\limits_{k=1}^{i(0,\omega)}(E+B_k)\Biggr)^{-1}T(t)=T(t)\Biggl(\rho-T(\omega)\prod\limits_{k=1}^{i(0,\omega)}(E+B_k)\Biggr)^{-1},\quad t\geq 0.
$$
Then we have
\begin{equation*}
\int\limits_0^{\omega}\|H(t,\tau)\|\,d\tau\leq C'_2
\equiv \left\{
\begin{aligned}
&M\Biggl\|\Biggl(\rho-T(\omega)\prod\limits_{k=1}^{i(0,\omega)}(E+B_k)\Biggr)^{-1}\Biggr\|\cdot\max\Biggl\{\Biggl\|\prod\limits_{k=1}^{i(0,\omega)}(E+B_k)\Biggr\|,1\Biggr\}\\
&\times\max\{\|T\|,1\}\frac{e^{\gamma\omega}-1}{\gamma},\quad\gamma\neq0;\\
&M\Biggl\|\Biggl(\rho-T(\omega)\prod\limits_{k=1}^{i(0,\omega)}(E+B_k)\Biggr)^{-1}\Biggr\|\cdot\max\Biggl\{\Biggl\|\prod\limits_{k=1}^{i(0,\omega)}(E+B_k)\Biggr\|,1\Biggr\}\\
&\times\max\{\|T\|,1\}\omega,\quad\gamma=0,
\end{aligned}\right.\end{equation*}
for any $t\in(0,\omega)$.
\end{lem}

\begin{proof}
We have
\begin{align*}
&\int\limits_0^{\omega}\|H(t,\tau)\|\,d\tau\\
&\leq\int\limits_0^t\|T\|\cdot \Biggl\|\Biggl(\rho-T(\omega)\prod\limits_{k=1}^{i(0,\omega)}(E+B_k)\Biggr)^{-1}\Biggr\|\cdot\|T(t-\tau)\|\cdot
\Biggl\|\prod\limits_{k=1}^{i(\tau,t)}(E+B_k) \Biggr\|\,d\tau\\
&+\int\limits_t^{\omega}\|T(t+\omega-\tau)\|\cdot \Biggl\|\prod\limits_{k=1}^{i(0,t)+i(\tau,\omega)}(E+B_k)\Biggr\|\cdot \Biggl\|\Biggl(\rho-T(\omega)\prod\limits_{k=1}^{i(0,\omega)}(E+B_k)\Biggr)^{-1}\Biggr\| \, d\tau\\
&\leq M\Biggl\|\Biggl(\rho-T(\omega)\prod\limits_{k=1}^{i(0,\omega)}(E+B_k)\Biggr)^{-1}\Biggr\|\max\Biggl\{\Biggl\|\prod\limits_{k=1}^{i(0,\omega)}(E+B_k)\Biggr\|,1\Biggr\}\\
&\times\Biggl(\int\limits_0^t\|T\|e^{\gamma(t-\tau)}\, d\tau+\int\limits_t^{\omega}e^{\gamma(t+\omega-\tau)}\, d\tau\Biggr).
\end{align*}
Now, we have two subcases: $\gamma\neq0$ and $\gamma=0$.
For $\gamma\neq0$, we have
\begin{align*}
&\int\limits_0^{\omega}\|H(t,\tau)\|\,d\tau\\
&\leq M\Biggl\|\Biggl(\rho-T(\omega)\prod\limits_{k=1}^{i(0,\omega)}(E+B_k)\Biggr)^{-1}\Biggr\|\max\Biggl\{\Biggl\|\prod\limits_{k=1}^{i(0,\omega)}(E+B_k)\Biggr\|,1\Biggr\}\\
&\times\Biggl(\int\limits_0^t\|T\|e^{\gamma(t-\tau)}\,d\tau+\int\limits_t^{\omega}e^{\gamma(t+\omega-\tau)}\,d\tau\Biggr)\\
&\leq M\Bigg\|\Biggl(\rho-T(\omega)\prod\limits_{k=1}^{i(0,\omega)}(E+B_k)\Biggr)^{-1}\Biggr\|\max\Biggl\{\Biggl\|\prod\limits_{k=1}^{i(0,\omega)}(E+B_k)\Biggr\|,1\Biggr\}
\end{align*}
\begin{align*}
&\times\Biggl(\|T\|\frac{e^{\gamma t}-1}{\gamma}+\frac{e^{\gamma\omega}-e^{\gamma t}}{\gamma}\Biggr)\\
&\leq M\Biggl\|\Biggl(\rho-T(\omega)\prod\limits_{k=1}^{i(0,\omega)}(E+B_k)\Biggr)^{-1}\Biggr\|\max\Biggl\{\Biggl\|\prod\limits_{k=1}^{i(0,\omega)}(E+B_k)\Biggr\|,1\Biggr\}\\
&\times\max\{\|T\|,1\}\Biggl(\frac{e^{\gamma t}-1}{\gamma}\Biggr).
\end{align*}
For $\gamma=0$, we have
\begin{align*}
&\int\limits_0^{\omega}\|H(t,\tau)\|\,d\tau\\
&\leq M\Biggl\|\Biggl(\rho-T(\omega)\prod\limits_{k=1}^{i(0,\omega)}(E+B_k)\Biggr)^{-1}\Biggr\|\max\Biggl\{\Biggl\|\prod\limits_{k=1}^{i(0,\omega)}(E+B_k)\Biggr\|,1\Biggr\}\\
&\times\Biggl(\int\limits_0^t\|T\|e^{\gamma(t-\tau)}\,d\tau+\int\limits_t^{\omega}e^{\gamma(t+\omega-\tau)}\,d\tau\Biggr)\\
&\leq M\Biggl\|\Biggl(\rho-T(\omega)\prod\limits_{k=1}^{i(0,\omega)}(E+B_k)\Biggr)^{-1}\Biggr\|\max\Biggl\{\Biggl\|\prod\limits_{k=1}^{i(0,\omega)}(E+B_k)\Biggr\|,1\Biggr\} \max\{\|T\|,1\}\omega.
\end{align*}
\end{proof}

\section{$(\omega,{\rho})$-periodic solutions of integro--differential impulsive problem}\label{impulse}

In this section, we consider the following abstract integro-differential impulsive equation:
\begin{equation*}
\left\{
\begin{aligned}
&\dot{y}(t)=Ay(t)+f\Biggl(t,y(t),\int\limits_{0}^tg(t,s,y(t))\, ds\Biggr), \quad t\neq{\tau}_k,\,\, k\in{\mathbb N};\\
&\Delta y|_{t={\tau}_k}=B_ky(t)+d_k,\quad\quad k\in{\mathbb N}.\end{aligned}\right.
\end{equation*}

We are going to prove, by using the Banach fixed point theorem and the Schauder fixed point theorem, that under certain conditions, we have the existence and uniqueness of solutions of (\ref{rav}).

\begin{thm}\label{thm1}
Let \emph{(A1)--(A7)} and \emph{(A9)} hold. If $0<LC_2<1$, where $L=(L_f(\nu)+M_1L_g(\nu))$ and $M_1$ is certain constant, then the equation \eqref{rav} has a unique $(\omega,\rho)$-periodic solution $y\in{\Phi}_{\omega,\rho}$, satisfying
\begin{align*}
\|y\|\leq\frac{\|f\|_0C_2+C_1}{1-LC_2},
\end{align*}
where
$$
\|f\|_0:=\Biggl\|f\Biggl(\cdot,0,\int_0^{\cdot}g(\cdot,s,0)\, ds\Biggr)\Biggr\|_{\infty}=\max\limits_{t\in[0,\omega]}\Biggl\|f\Biggl(t,0,\int\limits_0^tg(t,s,0)\, ds\Biggr)\Biggr\|.
$$
\end{thm}
\begin{proof}
Note that, if $y\in{\Phi}_{\omega,\rho}$, then $f(\cdot,y,\int_0^{\cdot}g(\cdot,s,y)\, ds)\in{\Phi}_{\omega,\rho}$. Keeping in mind Lemma \ref{lem1} and Lemma \ref{lem3}, we need to solve the fixed point problem
\begin{align*}
y(t)=\int\limits_{0}^{\omega}H(t,\tau)f\Biggl(\tau,y(\tau),\int\limits_0^{\tau}g(\tau,s,y(\tau))\,ds\Biggr)\,d\tau+\sum\limits_{i=1}^mH(t,{\tau}_i)d_i,\quad t\in[0,\omega].
\end{align*}
We define the operator $R$ on the space $\Psi$ as
\begin{align*}
(Ry)(t):=\int\limits_{0}^{\omega}H(t,\tau)f\Biggl(\tau,y(\tau),\int\limits_0^{\tau}g(\tau,s,y(\tau))\,ds\Biggr)\,d\tau+\sum\limits_{i=1}^mH(t,{\tau}_i)d_i.
\end{align*}
We will show that $R:\Psi\rightarrow\Psi$ is a contraction mapping. Let $y_1,\ y_2\in{\Psi}.$ Using Lemma \ref{est2}, we obtain:
\begin{align*}
& \bigl\|(Ry_1)(t)-(Ry_2)(t)\bigr\| \\
&\leq\int\limits_0^{\omega}\|H(t,\tau)\|\cdot  \Biggl\|f\Biggl(\tau,y_1(\tau),\int\limits_0^{\tau}g(\tau,s,y_1(\tau))\,ds\Biggr)-
f\Biggl(\tau,y_2(\tau),\int\limits_0^{\tau}g(\tau,s,y_2(\tau))\,ds\Biggr)\Biggr\|\,d\tau\\
&\leq \int\limits_0^{\omega}\|H(t,\tau)\|\cdot L_f(\nu)\Biggl(\|y_1-y_2\|\\
&+\Biggl\|\int\limits_0^{\tau}g(\tau,s,y_1(\tau))\, ds-\int\limits_0^{\tau}g(\tau,s,y_2(\tau))\, ds\Biggr\|\Biggr)\, d\tau\\
&\leq\int\limits_0^{\omega}\|H(t,\tau)\|\, d\tau\cdot L_f(\nu)\Biggl(\bigl\|y_1-y_2\bigr\|+\int\limits_0^{\tau}\bigl\|g(\tau,s,y_1(\tau))-g(\tau,s,y_2(\tau))\bigr\|\, ds\, d\tau\Biggr)\\
&\leq\int\limits_{0}^{\omega}\|H(t,\tau)\|\cdot L_f(\nu)\Biggl(\bigl\|y_1-y_2\bigr\|+L_g(\nu)\int\limits_{0}^{\tau}\bigl\|y_1-y_2\bigr\|\, ds\Biggr)\\
&\leq(L_f(\nu)+M_1L_g(\nu))\cdot \bigl\|y_1-y_2\bigr\|\cdot\int\limits_0^{\omega}\|H(t,\tau)\|\, d\tau\\
&\leq LC_2\cdot \bigl\|y_1-y_2\bigr\|,
\end{align*}
where $L:=L_f(\rho)+M_1L_g(\rho)$. So,
\begin{align*}
\bigl\|(Ry_1)-(Ry_2)\bigr\|\leq LC_2\bigl\|y_1-y_2\bigr\|,\quad 0<LC_2<1.
\end{align*}
Hence, the uniqueness of the solution of (\ref{rav}) follows by the Banach contraction mapping principle.
Moreover, we have
\begin{align*}
&\|y\|=\|Ry\|\\
&\leq\int\limits_{0}^{\omega}\|H(t,\tau)\|\cdot \Biggl\|f\Biggl(\tau,y(\tau),\int\limits_0^{\tau}g(\tau,s,y(\tau))\, ds\Biggr)\Biggr\|\, d\tau+\sum\limits_{i=1}^m\bigl\|H(t,{\tau}_i) d_i\bigr\|\\
&\leq\int\limits_0^{\omega}\|H(t,\tau)\|\cdot \Biggl\|f\Biggl(\tau,y(\tau),\int\limits_0^{\tau}g(\tau,s,y(\tau))\, ds\Biggr)-f\Biggl(\tau,0,\int\limits_0^{\tau}g(\tau,s,0)\Biggr)\, ds\Biggr\|\, d\tau
\end{align*}
\begin{align*}
&+\int\limits_0^{\omega}\|H(t,\tau)\|\cdot \Biggl\|f\Biggl(\tau,0,\int\limits_0^{\tau}g(\tau,s,0)\, ds\Biggr)\Biggr\|\, d\tau +\sum\limits_{i=1}^m\bigl\|H(t,{\tau}_i)d_i\bigr\|\\
&\leq\int\limits_{0}^{\omega} L_f(\nu)\Biggl(\|y\|+\int\limits_{0}^{\tau}\|g(\tau,s,y(\tau))-g(\tau,s,0)\|\, ds\Biggr)\, d\tau\\
&+\int\limits_0^{\omega}\|H(t,\tau)\|\cdot \Biggl\|f\Biggl(\tau,0,\int\limits_0^{\tau}g(\tau,s,0)\Biggr)\, ds\Biggr\|\, d\tau+\sum\limits_{i=1}^m\bigl\|H(t,{\tau}_i)d_i\bigr\|\\
&\leq\int\limits_0^{\omega}\|H(t,{\tau})\|\cdot L_f(\nu)\Biggl(\|y\|+M_1L_g(\nu)\|y\|\Biggr)\, d\tau+\|f\|_0\int\limits_0^{\omega}\|H(t,\tau)\|d\tau\\
&+\sum\limits_{i=1}^m\bigl\|H(t,{\tau}_i)d_i\bigr\|\leq\|y\|\Biggl(L_f(\nu)+M_1L_g(\nu)\Biggr)C_2+\|f\|_0C_2+C_1.
\end{align*}
We obtain:
\begin{align*}
\|y\|\leq\frac{\|f\|_0C_2+C_1}{1-\bigl(L_f(\nu)+M_1L_g(\nu)\bigr)C_2}=\frac{\|f\|_0C_2+C_1}{1-LC_2},
\end{align*}
so the proof is finished.
\end{proof}

Now we will state and prove the following result:

\begin{thm}
Let \emph{(A1)--(A6)} and \emph{(A8)--(A10)} hold. If $0<\beta C_2<1$, then \eqref{rav} has a $(\omega,\rho)$-periodic solution $y\in{\Phi}_{\omega,\rho}$.
\end{thm}

\begin{proof}
We consider the operator $R$ defined like in the proof of the previous theorem on $B_l=\{y\in\Psi \, :\, \|y\|\leq l\}$ and the constant $l$ is given by $l=\frac{\alpha C_2+C_1}{1-\beta C_2}$. We are going to prove the statement of the theorem in the following steps:\\
\indent {\bf Step 1.} We show that $R(B_l)\subset B_l$. Let $y\in B_l$ be arbitrary, and let $t\in[0,\omega]$. Then we have:
\begin{align*}
&\|Ry(t)\|\\
&\leq\int\limits_0^{\omega}\|H(t,\tau)\|\cdot \Biggl\|f\Biggl(\tau,y(\tau),\int\limits_{0}^{\tau}g(t,s,y(\tau))\Biggr)\, ds\Biggr\|\, d\tau +\sum\limits_{i=1}^m\bigl\|H(t,{\tau}_i)d_i\bigr\|\\
&\leq\beta\int\limits_0^{\omega}\bigl\|H(t,{\tau}_i)\bigr\|\cdot\|y(\tau)\|\, d\tau+\alpha\int\limits_0^{\omega}\|H(t,\tau)\|\, d\tau\\
&+\sum\limits_{i=1}^m\bigl\|H(t,{\tau}_i)\bigr\|\cdot\|d_i\|\leq\beta C_2\|y\|+\alpha C_2+C_1=l,
\end{align*}
implying $\|Ry(t)\|\leq l$, so $R(B_l)\subset B_l$ for any $t\in[0,\omega]$. Here the constants $\alpha$ and $\beta$ are from the assumption {(A8)}.\\
\indent {\bf Step 2.} We prove that the operator $R$ is continuous on $B_l$. Let $(y_n)$ be a Cauchy sequence such that $y_n\rightarrow y$, when $n\rightarrow\infty$ in $B_l$. Let \\ $f_n=f\Biggl(t,y_n(t),\int\limits_0^tg(t,s,y_n(t))\, ds\Biggr)$, $t\in[0,\omega]$ and $f=f\Biggl(t,y(t),\int\limits_0^tg(t,s,y(t))\, ds\Biggr)$. It is clear that $f_n\rightarrow f$, when $y_n\rightarrow y$, for any $t\in[0,\omega]$. Now,
\begin{align*}
& \bigl\|(Ry_n)(t)-(Ry)(t)\bigr\|\\
&\leq\int\limits_0^{\omega}\|H(t,\tau)\|\cdot\Biggl\|f\Biggl(t,y_n(t),\int\limits_0^tg(t,s,y_n(t))\, ds\Biggr)-f\Biggl(t,y(t),\int\limits_0^tg(t,s,y(t))\, ds\Biggr)\Biggr\|\, d\tau\\
&\leq\int\limits_{0}^{\omega}\|H(t,\tau)\|\cdot\|f_n-f\|\, d\tau\leq C_2\|f_n-f\|.
\end{align*}
Hence, $R$ is continuous operator on $B_l$.\\
\indent {\bf Step 3.} We show that $R(B_l)$ is relatively compact set. Since $R(B_l)\subset B_l$, it follows that $R(B_l)$ is uniformly bounded. Now, we show that the operator $R$ is an equicontinuous operator. For any $t_1, t_2\in[0,\omega]$ and $y\in B_l$, using {(A8)}, we have
\begin{align*}
& \bigl\|(Ry)(t_2)-(Ry)(t_1)\bigr\|\\
&\leq\int\limits_0^{\omega}\bigl\|H(t_2,\tau)-H(t_1,\tau)\bigr\|\cdot \Biggl\|f\Biggl(\tau,y(\tau),\int\limits_0^{\tau}g(\tau,s,y(\tau))\, ds\Biggr)\Biggr\|\, d\tau\\
&+\sum\limits_{i=1}^m\bigl\|H(t_2,{\tau}_i)-H(t_1,{\tau}_i)\bigr\|\cdot \|d_i\|\\
&\leq(\alpha+\beta\|y\|)\int\limits_0^{\omega}\bigl\|H(t_2,\tau)-H(t_1,\tau)\bigr\|\, d\tau+\sum\limits_{i=1}^{m}\bigl\|H(t_2,{\tau}-i)-H(t_1,{\tau}_i)\bigr\|\cdot\|d_i\|.
\end{align*}
We have:
\begin{equation*}
\bigl\|H(t_2,\tau)-H(t_1,\tau)\bigr\|
=\left\{
\begin{aligned}
& \Biggl\|T(t_2)\prod\limits_{k=1}^{i(0,t_2)}(E+B_k)\Biggl(\rho-T(\omega)\prod\limits_{k=1}^{i(0,\omega)}(E+B_k)\Biggr)^{-1}T(\omega-\tau)\\
&\prod\limits_{k=1}^{i(\tau,\omega)}(E+B_k)+T(t_2-\tau)\prod\limits_{k=1}^{i(\tau,t_2)}(E+B_k)\\
&-T(t_1)\prod\limits_{k=1}^{i(0,t_1)}(E+B-k)\Biggl(\rho-T(\omega)\prod\limits_{k=1}^{i(0,\omega)}(E+B_k)\Biggr)^{-1}\\
&\times T(\omega-\tau)\prod\limits_{k=1}^{i(\tau,\omega)}(E+B_k)-T(t_1-\tau)\prod\limits_{k=1}^{i(\tau,t_1)}(E+B_k)\Biggr\|,\\
&\quad 0<\tau<t_1<t_2;\\
&\Biggl\|T(t_2)\prod\limits_{k=1}^{i(0,t_2)}(E+B_k)\Biggl(\rho-T(\omega)\prod\limits_{k=1}^{i(0,\omega)}(E+B_k)\Biggr)^{-1} T(\omega-\tau)\\
&\prod\limits_{k=1}^{i(\tau,\omega)}(E+B_k)-T(t_1)\prod\limits_{k=1}^{i(0,t_1)}(E+B_k)\Biggl(\rho-T(\omega)\\
&\times\prod\limits_{k=1}^{i(0,\omega)}(E+B_k)\Biggr)^{-1}\cdot T(\omega-\tau)\prod\limits_{k=1}^{i(\tau,\omega)}(E+B_k)\Biggr\|,\quad t_1<t_2<\tau<\omega.
\end{aligned}\right.
\end{equation*}
If $0<\tau<t_1<t_2$, then we have
\begin{align*}
& \bigl\|H(t_2,\tau)-H(t_1,\tau)\bigr\|\\
& \leq \Biggl\|T(t_2)\prod\limits_{k=1}^{i(0,t_2)}(E+B_k)+T(t_2-\tau)\prod\limits_{k=1}^{i(\tau,t_2)}(E+B_k)T(\omega-\tau)\\
&\times\prod\limits_{k=1}^{i(\tau,\omega)}(E+B_k)+T(t_2-\tau)\prod\limits_{k=1}^{i(\tau,t_2)}(E+B_k)\\
&-T(t_1)\prod\limits_{k=1}^{i(0,t_1)}(E+B_k)+T(t_2-\tau)\prod\limits_{k=1}^{i(\tau,t_2)}(E+B_k)\\
&\times T(\omega-\tau)\prod\limits_{k=1}^{i(\tau,\omega)}(E+B_k)-T(t_1-\tau)\prod\limits_{k=1}^{i(\tau,t_1)}(E+B_k)\Biggr\|\\
&\leq \bigl\|T(t_2)-T(t_1)\bigr\|\cdot \Biggl\|\prod\limits_{k=1}^{i(0,\omega)}(E+B_k)^2\Biggr\|\cdot \Biggl\|\Biggl(\rho-T(\omega)\prod\limits_{k=1}^{i(0,\omega)}(E+B_k)\Biggr)^{-1}\Biggr\| \\
&\times Me^{\gamma(\omega-\tau)}+\bigl\|T(t_2-\tau)-T(t_1-\tau)\bigr\| \cdot \Biggl\|\prod\limits_{k=1}^{i(0,\omega)}(E+B_k)\Biggr\|.
\end{align*}
Using (A10), we obtain $\|H(t_2,\tau)-H(t_1,\tau)\|\rightarrow0$ as $t_2\rightarrow t_1$.\\
If $t_1<t_2<\tau<\omega$, we have
\begin{align*}
&\|H(t_2,\tau)-H(t_1,\tau)\|\\
&=\Biggl\|T(t_2)\prod\limits_{k=1}^{i(0,t_2)}(E+B_k)\Biggl(\rho-T(\omega)\prod\limits_{k=1}^{i(0,\omega)}(E+B_k)\Biggr)^{-1}T(\omega-\tau)\\
&\times\prod\limits_{k=1}^{i(\tau,\omega)}(E+B_k)-T(t_1)\prod\limits_{k=1}^{i(0,t_1)}(E+B_k)\\
&\times\Biggl(\rho-T(\omega)\prod\limits_{k=1}^{i(0,\omega)}(E+B_k)\Biggr)^{-1}T(\omega-\tau)\prod\limits_{k=1}^{i(\tau,\omega)}(E+B_k)\Biggr\|\\
&\leq\|T(t_2)-T(t_1)\|\cdot\Biggl\|\prod\limits_{k=1}^{i(0,\omega)}(E+B_k)^2\Biggr\|\\
&\times\Biggl\|\Biggl(\rho-T(\omega)\prod\limits_{k=1}^{i(0,\omega)}(E+B_k)\Biggr)^{-1}\Biggr\|Me^{\gamma(\omega-\tau)}.
\end{align*}
Again, using (A10), we obtain that $\|H(t_2,\tau)-H(t_1,\tau)\|\rightarrow0$ as $t_2\rightarrow t_1$.  Therefore, for any $t_1, t_2\in[0,\omega]$, we have $H(t_2,\tau)\rightarrow H(t_1,\tau)$, when $t_2\rightarrow t_1$. This implies that $\|(Ry)(t_2)-(Ry)(t_1)\|\rightarrow0$ when $t_2\rightarrow t_1$, hence the operator $R$ is an equicontinuous operator.\\
\indent {\bf Step 4.} We show that $R$ maps a bounded set into a precompact set in $X$. We consider the approximate operator $R_{\varepsilon}$ on $B_l$ as following
\begin{align*}
(R_{\varepsilon})(t)=\int\limits_0^{\omega}H(t-\varepsilon,\tau)f\Biggl(\tau,y(\tau),\int\limits_{0}^{\tau}g(\tau,s,y(\tau))\, ds\Biggr)\, d\tau+
\sum\limits_{i=1}^mH(t-\varepsilon,{\tau}_i)d_i,\quad t\in[0,\omega].
\end{align*}
Let
$
K=\{(Ry)(t)\, :\, t\in[0,\omega]\}
$
and
\begin{align*}
K_{\varepsilon}=T(\varepsilon)\bigl\{(R_{\varepsilon}y)(t)\, :\, t\in[0,\omega]\bigr\},\quad \varepsilon\in(0,\omega).
\end{align*}
Since $K$ is bounded and (A10) holds, we obtain that $K_{\varepsilon}$ is precompact.
Next,
\begin{align*}
& \bigl\|(R_{\varepsilon}y)(t)-(Ry)(t) \bigr\|\\
&\leq\int\limits_0^{\omega}\|H(t,\tau)-H(t-\varepsilon,\tau)\|\cdot \Biggl\|f\Biggl(\tau,y(\tau),\int\limits_{0}^{\tau}g(\tau,s,y(\tau))\, ds\Biggr)\Biggr\|\, d\tau\\
&+\sum\limits_{i=1}^m\bigl\|H(t,{\tau}_i)-H(t-\varepsilon,{\tau}_i)\bigr\|\cdot\|d_i\|\\
&\leq(\alpha+\beta l)\int\limits_{0}^{\omega}\|H(t,\tau)-H(t-\varepsilon,\tau)\|\, d\tau\\
&+\sum\limits_{i=1}^m\bigl\|H(t,{\tau}_i)-H(t-\varepsilon,{\tau}_i)\bigr\|\cdot\|d_i\|.
\end{align*}
Since for any $t_1,t_2\in(0,\omega)$, $H(t_2,\tau)\rightarrow H(t_1,\tau)$ as $t_2\rightarrow t_1$, we have
$
\|(R_{\varepsilon}y)(t)-(Ry)(t)\|\rightarrow 0
$
as $\varepsilon\rightarrow0$. Hence, $K$ can be approximated by a precompact set $K_{\varepsilon}$ with arbitrary accuracy. So, $K$ itself is a precompact set in $X$, i.e., $R$ maps a bounded set into a precompact set. The Arzel\`a--Ascoli theorem implies the compactness of the operator $R$, so the statement of the theorem follows by applying the Schauder fixed point theorem.
\end{proof}

\end{document}